\documentclass[11pt,reqno]{amsart}
\usepackage{amsfonts}
\usepackage{mathrsfs}
\usepackage{amsmath}
\usepackage{amssymb, color}
\usepackage{amsthm}
\usepackage{epsfig}
\usepackage{graphicx}
\usepackage{titletoc}
\usepackage{fullpage}
\usepackage{palatino,mathpazo}
\numberwithin{equation}{section}

\newtheorem{theorem}{Theorem}[section]

\newtheorem{remark}[theorem]{Remark}
\newtheorem{lemma}[theorem]{Lemma}

\theoremstyle{definition}

\theoremstyle{remark}

\def\XXint#1#2#3{{\setbox0=\hbox{$#1{#2#3}{\int}$}
        \vcenter{\hbox{$#2#3$}}\kern-.5\wd0}}

\def\XXint#1#2#3{{\setbox0=\hbox{$#1{#2#3}{\int}$ }
        \vcenter{\hbox{$#2#3$ }}\kern-.6\wd0}}

\begin{document}
\title[Existence and nonexistence of minimizer for
Thomas-Fermi-Dirac-von Weizs\"{a}cker model on lattice graph]{
Existence and nonexistence of minimizer for Thomas-Fermi-Dirac-von
Weizs\"{a}cker model on lattice graph}

\author[Y. Liu]{Yong Liu}
\address{Yong Liu, Department of Mathematics,  University of Science and Technology of China, Hefei, China}
\email{yliumath@ustc.edu.cn}

\author[J. Wang]{Jun Wang}
\address{\noindent Jun ~Wang,~Institute of Applied System Analysis, Jiangsu University,
    Zhenjiang, Jiangsu, 212013, P.R. China.}
\email{wangmath2011@126.com }

\author[K. Wang]{Kun Wang}
\address{\noindent Kun ~Wang,~Institute of Applied System Analysis, Jiangsu University,
    Zhenjiang, Jiangsu, 212013, P.R. China.}
\email{wangkun880304@163.com}

\author[W. Yang]{Wen Yang}
\address{\noindent Wen Yang, Department of Mathematics, Faculty of Science and Technology, University of
	Macau, Macau, P.R. China}
\email{wenyang@um.edu.mo}


\thanks{The research of Y. Liu is partially supported by NSFC No. 11971026 and No. 12141105. The research of J. Wang is partially supported by National
Key R$\&$D Program of China (2022YFA1005601), NSFC No. 11971202 and
Outstanding Young foundation of Jiangsu Province No. BK20200042. The research of W. Yang  is partially supported by National Key R\&D Program of China 2022YFA1006800, SRG 2023-00067-FST, NSFC No. 12171456 and NSFC No. 12271369.}

\begin{abstract}
The focus of our paper is to investigate the possibility of a
minimizer for the Thomas-Fermi-Dirac-von Weizs\"{a}cker model on the
lattice graph $\mathbb{Z}^{3}$. The model is described by the
following functional:
\begin{equation*}
E(\varphi)=\sum_{y\in\mathbb{Z}^{3}}\left(|\nabla\varphi(y)|^2+
(\varphi(y))^{\frac{10}{3}}-(\varphi(y))^{\frac{8}{3}}\right)+
\sum_{x,y\in\mathbb{Z}^{3}\atop ~\ y\neq x\hfill}\frac{{\varphi}^2(x){\varphi}^2(y)}{|x-y|},
\end{equation*}
with the additional constraint that $\sum\limits_{y\in\mathbb{Z}^{3}}
{\varphi}^2(y)=m$ is sufficiently small. We also prove the
nonexistence of a minimizer provided the mass $m$ is adequately large.
Furthermore, we extend our analysis to a subset $\Omega \subset
\mathbb{Z}^{3}$ and prove the nonexistence of a minimizer for the
following functional:
\begin{equation*}
E(\Omega)=|\partial\Omega|+\sum_{x,y\in\Omega\atop ~y\neq
x\hfill}\frac{1}{|x-y|},
\end{equation*}
under the constraint that $|\Omega|=V$ is sufficiently large.

\end{abstract}

\maketitle {\bf Keywords}: Variational methods, Lattice graph;
Nonlocal term; Minimizer solution.

\

\

\section{Introduction}

The aim of this paper is to investigate whether a minimizer exists
for two mass-constrained variational problems on a lattice graph.
Specifically, we are interested in the
Thomas-Fermi-Dirac-von Weizs\"{a}cker (TFDW) functional
\begin{equation}\label{a-3}
	E(\varphi)=\int_{\mathbb{R}^{3}}\left(|\nabla\varphi|^2+
	\varphi^{\frac{10}{3}}-\varphi^{\frac{8}{3}}\right) dx
	+\int_{\mathbb{R}^{3}}
	\int_{\mathbb{R}^{3}}\frac{{\varphi}^2(x){\varphi}^2(y)}{|x-y|}dxdy
\end{equation}
and the Gamow's liquid drop functional
\begin{equation}\label{a-4}
E(\Omega)=|\partial\Omega|
+\int_{\Omega}\int_{\Omega}\frac{1}{|x-y|}dxdy,
\end{equation}
where $\Omega\subset\mathbb{R}^{3}$ is a bounded set.

The functional \eqref{a-3} is commonly known in the Physics
literature as the Thomas-Fermi-Dirac-von Weizs\"{a}cker (TFDW) model
for electrons in the absence of an external potential(the electron
density is represented by $\varphi^2=\sigma$). The ionization
conjecture in quantum mechanics is also associated with this model.
To introduce it, we define the minimized variational problem
\begin{equation}\label{aj5}
I_Z^{TFDW}:=\inf\left\{E^{TFDW}(\varphi):\varphi\in
H^1(\mathbb{R}^{3}),\int_{\mathbb{R}^{3}} {\varphi^{2}}dx=N\right\},
\end{equation}
where
\begin{equation}\label{aj6}
\begin{split}
E^{TFDW}(\varphi)&=C^{W}\int_{\mathbb{R}^{3}}|\nabla\varphi|^2dx+C^{TF}\int_{\mathbb{R}^{3}}
\varphi^{\frac{10}{3}}-C^{D}\int_{\mathbb{R}^{3}}\varphi^{\frac{8}{3}}
dx\\
&\quad-Z\int_{\mathbb{R}^{3}}\frac{\varphi^2}{|x|}dx
+\int_{\mathbb{R}^{3}}
\int_{\mathbb{R}^{3}}\frac{{\varphi}^2(x){\varphi}^2(y)}{|x-y|}dxdy.
\end{split}
\end{equation}
The ionization conjecture states that
\medskip

\noindent {{\bf Ionization conjecture}: The number of electrons that can be bound to an atomic nucleus of charge $Z$ cannot exceed $Z+1$.} Mathematically, if $N\leq Z+1$, then $I_Z^{TFDW}$ has a minimizer.
\medskip

\noindent For more details statements about this problem one can refer to
\cite[Problem 9]{Simon-2000} and \cite[Chapter
12]{Lieb-Seiringer-2010}. Although experimental data shows that a
neutral atom has the capacity to bind up to two extra electrons, it
remains challenging to establish a rigorous justification for this
observation based on the fundamental principles of quantum
mechanics. Over the last few years, a significant number of
mathematicians have dedicated their efforts to analyzing the problem
\eqref{aj5}-\eqref{aj6}. In the paper \cite{P. Lions-1987}, the
author proved the existence of solution to the classical TFDW
problem \eqref{aj5} for $N\leq Z$ and Le Bris \cite{Le Bris-1993}
extended it to $N\leq Z+C$ for some $C>0$. Fefferman and Seco et. al in \cite{Fefferman-1990,Seco-Solovej-1990} proved that if
$Z\rightarrow\infty$, then $N_c(Z)\leq Z+O(Z^{5/7})$. In  \cite{Lieb-PhysicA-1984,Nam-2012}, they proved that for all
$Z\geq1$, $N_c(Z)<\min(2Z+1, 1.22Z+3Z^{1/3})$ holds true. In the
paper \cite{Solovej-Annals-2003}, the author prove that there exists
a universal constant $C>0$ such that if $N>Z+C$, then $E^{HF}(N)$
has no minimizer, where $0\leq\gamma\leq1$, $\gamma^2=\gamma$ and
$Tr\gamma=N$
$$E^{HF}(N)=\inf_{Tr\gamma=N}\left( Tr(-\Delta-Z|x|^{-1})\gamma)
+\frac{1}{2}\int_{\mathbb{R}^{3}}
\int_{\mathbb{R}^{3}}\frac{\gamma(x;x)\gamma(y;y)-\gamma(x;y)}{|x-y|}dxdy\right\}$$
and $\gamma$ is a density matrix with eigenfunctions $u_j$ and
corresponding eigenvalues $\nu_j$ on either $L^2(\mathbb{R}^3)$ or
$L^2(\mathbb{R}^3; \mathbb{C}^2)$ we shall write
\begin{equation*}
Tr[-\Delta\gamma]=\Sigma_{j}\nu_j\int_{\mathbb{R}^{3}}|\nabla
u_j|^2dx.
\end{equation*}
In the recent papers \cite{Nam-Bosch-2017,Frank-Nam-2018} proved
that there exists a constant $C>0$ such that for all $Z>0$, if
$I_Z^{TFDW}$ has a minimizer, then $N\leq Z+C$. For more results on
this direction one can refer to the references
\cite{Benguria-1981,Fermi-1927,Lieb-1981,Lieb-1984,Lieb-1988-CMP,Nam-2022-EMS,
Ruiz-2010,Sanchez-2004,Solovej-Invention-2003,L.H.Thomas-1927,Weizscker-1935}
and references therein.

To the best of our knowledge the ionization conjecture was still not
complete solved in the continue case. For the special case $Z=0$,
the ionization conjecture can also be interpreted as the vacuum can
only bind a finite number of electrons in the TFDW model. The
work \cite{J. Lu-2014} established that for a large
mass, the functional $I_0^{TFDW}$ given by
\begin{equation}\label{aj7}
I_0^{TFDW}=\inf_{\varphi\in
H^1(\mathbb{R}^{3}),\int_{\mathbb{R}^{3}}\varphi^2dx=m}E(\varphi)
\end{equation}
does not have a minimizer. The first objective of this paper is to
examine the existence and nonexistence of a minimizer of \eqref{aj7}
on the lattice graph $\mathbb{Z}^{3}$.

On the other hand, for $\Omega\in\mathbb{R}^{3}$, we consider the
Gamow's liquid drop model
\begin{equation}\label{a-8}
E_V=\inf\left\{E(\Omega):|\Omega|=V\right\},
\end{equation}
where $E(\Omega)$ is given in \eqref{a-4}. If the Coulomb term is
not present in \eqref{a-3} and \eqref{a-4}, When the volume
constraint $m$ is large, a Modica-Mortola type result establishes a
connection between the two functionals via Gamma convergence, as
stated in \cite{Modica-1987}. Gamow initially proposed problem
\eqref{a-8} in \cite{Gamow-1930} with the intention of uncovering
fundamental properties of atoms and offering a straightforward
nuclear fission model. Out of all the measurable sets with given
volume, a ball minimizes the perimeter term while maximizing the
Coulomb term. This highlights the significance of examining the
interplay between the perimeter term and Coulomb term. In 
\cite{Choksi-2010}, the authors proved the existence of a minimizer
of \eqref{a-4} as $V$ is small. The paper \cite{Knupfer-2014} proved
the nonexistence of the minimizer as $V$ is large by a deep
techniques in geometric measure theory. Also, Lu and Otto \cite{J.
Lu-2014} proved the nonexistence results by using different
arguments. Furthermore, this type of model is present in the study
of diblock copolymer melt models, Thomas-Fermi type models, and
various other systems, for instance, see
\cite{Alberti-2009,Care-1975,Chen-Khachaturyan-1993,
Choksi-2011,Cicalese-2013,Frank-Lieb-2015, Knupfer-2013,
Muratov-2002,Muratov-2010, Ohta-2009} and the references therein.

In this paper, we investigate the existence and nonexistence of a
minimzer of the energy \eqref{a-3} and \eqref{a-4} on the lattice
graph $\mathbb{Z}^{3}.$ Recently, there has been growing interest among scholars from
diverse fields in graph analysis. Specifically, researchers have
explored the extension of Sobolev inequalities and sharp constants
to finite graphs, as evidenced by studies such as
\cite{Chung-1995,Hua-Li-2021,Nagai-2009,Ostrovskii-2005}. On the
other hand, the study of partial differential equations on graphs is
a fascinating area of research. Notably, the papers
\cite{Grigoryan-Lin-Yang-2016,Grigoryan-Lin-Yang-PDE-2016,Grigoryan-Lin-Yang-2017}
have contributed significantly to this field through a series of
works. Their research employs variational methods to establish the
existence of Yamabe type equations and some nonlinear elliptic
equations on graph. The heat equation has also been a topic of
interest, with several researchers exploring its existence,
uniqueness, and blow-up properties. Notable studies in this area
include \cite{Liu-Chen-Yu-2016, Wojciechowski-2009}. For further
related research, interested readers may refer to
\cite{Grigoryan-2018-Book,
	Lin-Yau-2019-Adv,Lin-Yau-2019-Crell,Huang-Yau-2020,Huang-Wang-Yang-2021, Hua-De-2015,
Keller-2021,Sun-2022-Adv} and the
corresponding references.

To present our findings lucidly, the following is a description of
the basic setting for graphs. Consider a simple, undirected, and
locally finite graph $\mathbb{G} = (\mathbb{V}, \mathbb{E})$, where $\mathbb{V}$ represents the set of
vertices and $\mathbb{E}$ represents the set of edges. Vertices $x$ and $y$
are considered neighbours, written as $x\sim y$, if there is an edge connecting them, meaning that $\{x, y\}\in \mathbb{E}$. A graph is considered
locally finite if each vertex has a finite number of neighbours.
Denote the set of functions on $\mathbb{G}$ by $C(\mathbb{V})$. The Laplacian on
$\mathbb{G}=(\mathbb{V},\mathbb{E})$ is defined as, for any function $u\in C(\mathbb{V})$ and $x\in \mathbb{V}$,
\begin{equation*}
\Delta u(x)=\sum_{y\in \mathbb{V}, ~y\sim x}\left(u(x)-u(y)\right).
\end{equation*}
For any function $u,v \in C(\mathbb{V})$, we can define the associated gradient by
\begin{equation*}
	\Gamma(u, v)(x):=\sum_{y\in\mathbb{V}, ~y\sim x} \frac{1}{2}(u(y)-u(x))(v(y)-v(x)).
\end{equation*}
In particular, let $\Gamma(u)=\Gamma(u, u)$ for
simplicity. We denote the length of its gradient by
\begin{equation*}
	|\nabla u|(x):=\sqrt{\Gamma(u)(x)}=\frac{1}{2}\left(\sum_{y\in\mathbb{V},~ y \sim x}
	(u(y)-u(x))^{2}\right)^{1 / 2}.
\end{equation*}
Let $\mu$ be the counting measure on $\mathbb{V}$, i.e., for any subset
$A\subset \mathbb{V}$, $\mu(A)=\#\{x: x\in A\cap\mathbb{V}\}$. For any function $f$ on
$\mathbb{V}$, we write
\begin{equation*}
	\int_{\mathbb{V}} f d \mu:=\sum_{x \in \mathbb{V}} f(x),
\end{equation*}
whenever it makes sense. The corresponding energy is defined as
\begin{equation*}
\mathcal{E}(u):=\int_{\mathbb{V}}|\nabla u|^2d\mu=\frac12\sum_{x\in\mathbb{V}}\sum_{y\in\mathbb{V},~ y \sim x}(u(y)-u(x))^2.
\end{equation*}
By definition, $H^1(\mathbb{V})$ is a Hilbert space with the inner product given by
\begin{equation*}
\langle u,v\rangle :=\int_{ \mathbb{V}}\left(\Gamma(u,v)+uv\right)d\mu=\frac12\sum_{x\in\mathbb{V}}\sum_{y\in\mathbb{V},~ y \sim x}(u(y)-u(x))(v(y)-v(x))+\sum_{y\in\mathbb{V}}u(y)v(y).
\end{equation*}
The lattice graph $\mathbb{Z}^{3}$ is a very representative locally finite graph, which consists of a group of vertices that are defined as
follows:  the set of vertices
\begin{equation}
	\mathbb{V}=\left\{x=(x_1,x_2,x_3):x_j\in\mathbb{Z},~1\leq j\leq3\right\}
\end{equation}
and the set of edges
\begin{equation}
	\mathbb{E}=\left\{\{x,y\}:x,y\in\mathbb{Z}^3,~\sum_{j=1}^3|x_j-y_j|=1\right\}.
\end{equation}
For a subset $\Omega$ of $\mathbb{Z}^{3}$, we can define its volume and boundary area based on the counting measure as follows: 
\begin{equation*}
|\Omega|:=\#\{x: x \in
\Omega\} \quad \mathrm{and}\quad
|\partial \Omega| :=\left|\left\{ {y\in \Omega  ,~\exists x \notin \Omega\,\,\,\textrm{such that}\ \{x, y\} \in E} \right\}\right|.
\end{equation*}
Let $C(\mathbb{Z}^{3})$ be the set of real-valued functions on
$\mathbb{Z}^{3}$ and $\ell^p(\mathbb{Z}^{3})$, $p\in[1,\infty]$, be
the space of $\ell^p$ summable functions on $\mathbb{Z}^{3}$ w.r.t. the counting measure. We write $\|\cdot\|_p$ as the $\ell^p(\mathbb{Z}^{3})$ norm,
i.e.,
\begin{equation*}
	\|u\|_p=
	\begin{cases}
		\sum\limits_{x\in \mathbb{Z}^{3}}|u(x)|^p,\ &1\leq p<\infty,\\
		\sup\limits_{x\in \mathbb{Z}^{3}}|u(x)|, &p=\infty.
	\end{cases}
\end{equation*}
By definition, we know that
\begin{equation*}
	\|u\|_{H^1(\mathbb{Z}^{3})}^2=\sum_{y\in\mathbb{Z}^{3}}(|\nabla u(y)|^2+|u(y)|^2).
\end{equation*}
From \cite[The inequality (5)]{Hua-Xu-2023}, we have the following
equivalent norm
\begin{equation}\label{H-2}
	\|u\|_{H^1(\mathbb{Z}^{3})}\simeq
	\|u\|_{\ell^2(\mathbb{Z}^{3})}=\sum_{y\in\mathbb{Z}^{3}}|u(y)|^2.
\end{equation}
In the lattice graph setting, the correponding version of energy functional \eqref{a-3} and \eqref{a-4} are the following respectively
\begin{equation}\label{a-1}
	E(\varphi)=\sum_{y\in\mathbb{Z}^{3}}\left(|\nabla\varphi(y)|^2+
	\varphi^{\frac{10}{3}}(y)-\varphi^{\frac{8}{3}}(y)\right)+
	\sum_{x,y\in\mathbb{Z}^{3}\atop ~\ y\neq x\hfill}\frac{{\varphi}^2(x){\varphi}^2(y)}{|x-y|},
\end{equation}
and
\begin{equation}\label{a-2}
	E(\Omega)=|\partial\Omega|+\sum_{x,y\in\Omega\atop ~y\neq
		x\hfill}\frac{1}{|x-y|}.
\end{equation}
\medskip

The first two results concern the functional \eqref{a-1} by establishing the existence of minimizer for small mass and non-existence of minimizer for large mass.

\begin{theorem}\label{theorem1-3}
 If $m$ is small enough,    then the variational problem
\begin{equation}\label{a-5}
I_0^{TFDW}=\inf_{\sum\limits_{y\in\mathbb{Z}^{3}}
	{\varphi}^2(y)=m}E(\varphi)
\end{equation}
has a minimizer.
\end{theorem}

\begin{theorem}\label{theorem1-2}
There exists a positive constant $m_0$  such that the variational
problem
\begin{equation}\label{a-6}
I_0^{TFDW}=\inf\limits_{\sum\limits_{y\in\mathbb{Z}^{3}}
	{\varphi}^2(y)=m}E(\varphi)
\end{equation}
does not have a minimizer if $m>m_0$.
\end{theorem}

\begin{remark}
Drawing on the perspective from the continuous case presented in \cite{J. Lu-2014}, Theorem \ref{theorem1-2} can be understood as asserting that, from a physical standpoint, the vacuum can only bind a limited number of electrons in the TFDW model on a lattice graph. In turn, this result provides affirmation for the "ionization conjecture" in the specific scenario of $Z=0$ in the lattice graph case.
\end{remark}

Consider \eqref{a-2}, we have the following non-existence result

\begin{theorem}\label{theorem1-1}
For a subset $\Omega\subset\mathbb{Z}^{3}$, there exists a positive
constant $V_0$, such that the variational problem
\begin{equation}\label{a-7}
I_0^{TFDW}(\Omega)=\inf\limits_{|\Omega|=V}E(\Omega)
\end{equation}
does not have a minimizer if $V>V_0$.
\end{theorem}

Theorems \ref{theorem1-3}-\ref{theorem1-1} can be viewed as counterparts to the results presented in \cite{Choksi-2010} and \cite{J. Lu-2014} in the discrete setting. Although the results obtained in this work align with those of the continuous case and are conceptually natural when extending from Euclidean space to lattice graph. Nevertheless, the proof methodologies differ due to the distinct characteristics of graph structures in comparison to continuous spaces. Specifically, in the proof of Lemma \ref{lemma4.1}, while a similar conclusion as \eqref{4-2} can be easily established using the sophisticated scaling function $\psi_\eta$ in the continuous setting, as demonstrated in \cite{J. Lu-2014}. It seems that equivalent techniques are not available in the lattice graph setting. Instead, we introduce the test function \eqref{d-2} as a replacement for the scaling function $\psi_\eta$  in order to establish a similar conclusion in the discrete case. This highlights a notable distinction between the discrete and continuous cases.
\medskip

The remainder of this paper is structured as follows. In section 2,
we introduce some fundamental notations and present the formula for
the measure of the ball on the lattice graph $\mathbb{Z}^{3}$. In
section 3, we establish the proof for Theorem \ref{theorem1-3}.
Section 4 is dedicated to proving Theorem \ref{theorem1-2}. Finally,
in section 5, we provide the proof for Theorem \ref{theorem1-1}.
\medskip

\begin{center}
	\textbf{Notation}
\end{center}

Throughout the paper the letter $C$ will stand for positive constants which are allowed to vary among different formulas and even within the same lines. 
\medskip

\section{Preliminary results}
This section is devoted to provide some definitions of the integer lattice graph $\mathbb{Z}^3$ and useful conclusions appeard in previous works.  A subset $\Omega\subset\mathbb{Z}^3$ is considered connected if there exists a path between any two vertices $x$ and $y$ in $\Omega$, denoted by
$\{x_i\}_{i=1}^{l}\subset\Omega$, such that $x_1=x$, $x_l=y$ and $x_j\sim x_{j+1}$ for any $1\leq j\leq l-1$. The distance
$d(x,y)$ between two vertices $x$ and $y$ in $\mathbb{Z}^{3}$ is
defined as the infimum of $l$ such that $x=x_1\sim x_2
\sim\cdots\sim x_l=y$, i.e.,
\begin{equation*}
d(x,y)=\inf\{l:x=x_1\sim x_2 \sim\cdots\sim x_l=y\}.
\end{equation*}
The diameter of the $\Omega$ is defined as
\begin{equation*}
\textrm{diam}(\Omega):=\sup_{x,y \in \mathbb{Z}^{3}} d(x,y).
\end{equation*}
For any positive integer $R$, the ball on the integer lattice graph is defined by
\begin{equation*}
B_R(0)=\left\{x \in \mathbb{Z}^{3},~ d(x,0)\leq R  \right\},
\end{equation*}
and its boundary is
\begin{equation*}
\partial B_R(0)=\left\{x \in \mathbb{Z}^{3}, ~d(x,0)=R  \right\}.
\end{equation*}
In the following lemma we compute the volume and surface area of $B_R(0)$ in the discrete setting

\begin{lemma}\label{lemma2-2}
For any positive integer $R$, then we have $|\partial B_R|=4R^2+2$
and $|B_R|=\frac{4R^3+6R^2+8R+3}{3}$.
\end{lemma}

\begin{proof}
According to the definition of the distance for the lattice graph, we see that for any $x=(x_1,x_2, x_3)\in \partial B_R(0)$, it holds that	
$$|x_1|+|x_2|+|x_3|=R.$$ 
Since $\partial B_R(0)$ is symmetric, we divide it into three parts as depicted in Figure \ref{fig:2}.
\begin{figure}[htbp]
  \centering
  \includegraphics[width=0.6\textwidth]{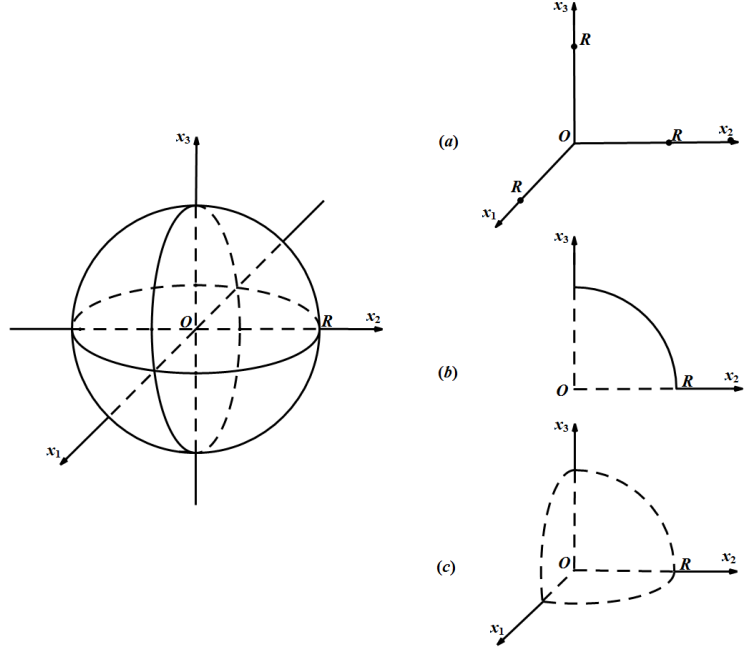}
  \caption{The sketch of $\partial B_R$.}\label{fig:2}
\end{figure}

\noindent ($i$) If $x$ lies at the intersection of the axis and the boundary of the circle with radius $R$, as presented in Figure \ref{fig:2} $(a)$, then there is exactly one value of $x_i$ that is not equal to zero. There are six possible options:
\begin{equation*}
(\pm R,0,0),(0,\pm R,0),(0,0,\pm R).
\end{equation*}

\noindent ($ii$) If $x$ lies at the intersection of the plane and the boundary of the sphere with radius $R$, as illustrated in Figure \ref{fig:2} $(b)$, then $x_i$ can only be zero for one particular index $i$. Without loss of generality, we assume that $x_1=0$, which implies $|x_2|+|x_3|=R$ and $x_2,x_3\neq0$. It can be easily demonstrated that there are $4C_{R-1}^1=4(R-1)$ possible choices in this case. While if $x_2=0$ or $x_3=0$, and the remaining two components are nonzero, we can obtain an additional $8(R-1)$ points.
\medskip

\noindent ($iii$) If $x$ is located at the intersection of the quadrant and the boundary of the sphere with radius $R$, as depicted in Figure \ref{fig:2} $(c)$, none of the components $x_i$ can be zero. In other words, we have $|x_1|+|x_2|+|x_3|=R$, where $x_1,x_2,x_3\neq0$. In this case, there are $8C_{R-1}^2=4(R^2-3R+2)$ points that satisfy this condition.
\medskip

Summarize the discussion from ($i$) to ($iii$), we have
\begin{equation*}
\begin{split}
|\partial B_R|
=\left|\{x \in \mathbb{Z}^{3}\mid |x_1|+|x_2|+|x_3|=R\}\right|
=6+12(R-1)+4(R^2-3R+2)=4R^2+2.
\end{split}
\end{equation*}
Therefore, one gets
\begin{equation*}
\begin{split}
|B_R|&=1+\sum\limits_{\rm{i}=1}^R|\partial B_i|=1+\sum\limits_{\rm{i}=1}^R {(4i^2+2)}=1+2R+4\sum\limits_{\rm{i}=1}^Ri^2=\frac{4R^3+6R^2+8R+3}{3}.
\end{split}
\end{equation*}
This completes the proof.
\end{proof}

Next let us recall the following two discrete Br\'ezis-Lieb type
Lemmas.

\begin{lemma}\label{lemmaB-L}
	$($\cite[Lemma 10]{Hua-Li-2021}$)$ Let $\Omega \subset \mathbb{Z}^{N}$
	be a domain and  $\{u_n\} \subset \ell^{q}(\Omega)$ with
	$1\leq q<\infty$. If $\{u_n\}$ is bounded in $\ell^{q}(\Omega)$ and
	$u_{n} \rightarrow u$ pointwise on $\Omega$ as $n \to \infty $, then
	\begin{equation}\label{B-L}
		\lim _{n \rightarrow
			\infty}\left(\|u_n\|_{\ell^{q}(\Omega)}^{q}-\|u_n-u\|_{\ell^{q}(\Omega)}^{q}\right)=\|u\|_{\ell^{q}(\Omega)}^{q}.
	\end{equation}
\end{lemma}

\begin{lemma}\label{lemma-nonlocal}$($\cite[Lemma 3.7]{Wang-2023}$)$
	Let $1 \leq p<\infty$ and the sequence $\{u_n\}$ is  bounded in
	$\ell^{\frac{2Np}{N+ \alpha}}(\mathbb{Z}^{N})$. Suppose that $u_n
	\rightarrow u$ pointwise on ${\mathbb{Z}^{N}}$ as $n \to \infty $,
	then
	\begin{equation}\label{c-11}
		\begin{split}
			\lim\limits_{n\to\infty} \left(
			\sum_{x,y \in {\mathbb{Z}^N} \hfill\atop ~\ y\ne
				x\hfill}\frac{|u_n(x)|^p|u_n(y)|^p}{|x-y|^{N-\alpha}}
			-\sum_{x,y \in {\mathbb{Z}^N} \hfill\atop ~\ y\ne
				x\hfill}\frac{|u_n(x)-u(x)|^p|u_n(y)-u(y)|^p}{|x-y|^{N-\alpha}}\right)=\sum_{x,y \in {\mathbb{Z}^N} \hfill\atop ~\ y\ne
				x\hfill}\frac{|u(x)|^p|u(y)|^p}{|x-y|^{N-\alpha}}.
		\end{split}
	\end{equation}
\end{lemma}

\noindent The subsequent two important inequalities play a crucial role in our argument 
\begin{lemma}$($\cite[Lemma 7]{Hua-Li-2021},~\cite[Lemma 2.11]{Huang-Li-Yin-2014}$)$
	If $u\in \ell^p(\mathbb{Z}^{N})$, then for all $q\geq p$,
	\begin{equation}\label{p-q}
		\|u\|_{\ell^q(\mathbb{Z}^{N})}\leq \|u\|_{\ell^p(\mathbb{Z}^{N})}.
	\end{equation}
	
\end{lemma}
\begin{lemma}\label{lemma2-1}$($\cite[Theorem 1.1]{Huang-Li-Yin-2014}$)$
	$($Discrete Hardy-Littlewood-Sobolve Inequality$)$ Let $0<\alpha <N$,
	$1<r,s<\infty$ and $\frac{1}{r} + \frac{1}{s} + \frac{{N - \alpha
	}}{N} \ge 2$. Assume $f \in \ell^r(\mathbb{Z}^{N})$ and $g \in
	\ell^s(\mathbb{Z}^{N})$. Then there exists a positive  constant $C$
	depending only on $r,s,\alpha$ such that
	\begin{equation}\label{b-1}
		\sum_{ x,y \in\mathbb{Z}^{N} \hfill\atop ~\ y\ne x\hfill}
		{\frac{{f(x)g(y)}}{|x-y|^{N-\alpha}}} \le C{\| f
			\|_{{\ell^r}(\mathbb{Z}^{N})}}{\| g\|_{{\ell^s}(\mathbb{Z}^{N})}}.
	\end{equation}
\end{lemma}

\section{The existence of a minimizer for \eqref{a-1}}

This section is dedicated to presenting the proof of Theorem
\ref{theorem1-3}. For the convenience, we denote
\begin{equation*}
I_m:=\inf\left\{E(\varphi):\varphi\in
H^1(\mathbb{Z}^{3}),\sum_{y\in\mathbb{Z}^{3}} {\varphi^{2}(y)}=m\right\}
\end{equation*}
and
$$F(s):=s^{\frac{5}{3}}-s^{\frac{4}{3}},\,\,\,D(f,g)=\sum_{x,y\in\mathbb{Z}^{3}\atop~\ x\neq y\hfill}\frac{f^2(x)g^2(y)}{|x-y|}.$$

We begin by establishing a property of strict subadditivity for $I_m$.

\begin{lemma}\label{lemma5.1}
If $m>0$ is small enough, then for any $m_1\in(0,m)$ it holds that
\begin{equation}\label{q-1}
I_m<I_{m_1}+I_{m-m_1}.
\end{equation}
\end{lemma}

\begin{proof}
First, we claim that if $\theta>1$, then the equality
\begin{equation}\label{q-2}
I_\theta m<\theta I_m
\end{equation}
holds for $m>0$ is small enough. Indeed, let $\varphi$ satisfying
$I_m\leq E(\varphi)\leq I_m+\varepsilon$ and
$\|\varphi\|_{\ell^2}^2=m$. By \eqref{p-q} and the
Hardy-Littlewood-Sobolve inequality \eqref{b-1}, we have
\begin{equation*}
\begin{split}
I_{\theta m}\leq E\left(\sqrt\theta \varphi\right)
=~&\theta\sum_{y\in\mathbb{Z}^{3}}|\nabla\varphi(y)|^2+ \theta^\frac{5}{3}\sum_{y\in\mathbb{Z}^{3}}\varphi^{\frac{10}{3}}(y)-\theta^\frac{4}{3}\sum_{y\in\mathbb{Z}^{3}}\varphi^{\frac{8}{3}}(y)+ \theta^2D(\varphi^2,\varphi^2)\\
=~&\theta \sum_{y\in\mathbb{Z}^{3}}\left(|\nabla\varphi(y)|^2+
\varphi^{\frac{10}{3}}(y)-\varphi^{\frac{8}{3}}(y)\right) +
\theta D(\varphi^2,\varphi^2)
+(\theta^\frac{5}{3}-\theta)\sum_{y\in\mathbb{Z}^{3}}\varphi^{\frac{10}{3}}(y)\\
&+(\theta-\theta^\frac{4}{3})\sum_{y\in\mathbb{Z}^{3}}\varphi^{\frac{8}{3}}(y)+(\theta^2-\theta)D(\varphi^2,\varphi^2)\\
\leq ~&\theta
(I_m+\varepsilon)+(\theta^\frac{5}{3}-\theta)\|\varphi\|_{\ell^2(\mathbb{Z}^3)}^\frac{10}{3}+(\theta-\theta^\frac{4}{3})
\|\varphi\|_{\ell^2(\mathbb{Z}^3)}^\frac{8}{3}+(\theta^2-\theta)C\|\varphi\|_{\ell^\frac{12}{5}(\mathbb{Z}^3)}^4\\
\leq ~&\theta
(I_m+\varepsilon)+(\theta^\frac{5}{3}-\theta)m^\frac{5}{3}+(\theta-\theta^\frac{4}{3})
m^\frac{4}{3}+(\theta^2-\theta)Cm^2.
\end{split}
\end{equation*}
where $C$ is a positive constant. Then we can choose $m$ small enough such that
$$(\theta^\frac{5}{3}-\theta)m^\frac{5}{3}+(\theta-\theta^\frac{4}{3})
m^\frac{4}{3}+(\theta^2-\theta)Cm^2<0.$$ 
As a consequence, we get that 
$$I_{\theta m}<\theta I_m.$$ 
We take $\theta>1$ such that $\theta m_1=m$. Then
$m-m_1=\frac{\theta-1}{\theta}m$. From \eqref{q-2} we get that
\begin{equation}\label{q-3}
I_m=I_{\theta m_1}<\theta I_{m_1}
\end{equation}
and
\begin{equation}\label{q-4}
I_m=I_{\frac{\theta}{\theta-1}
(m-m_1)}<\frac{\theta}{\theta-1}I_{m-m_1}.
\end{equation}
Combining with \eqref{q-3} and \eqref{q-4}, it follows that
\begin{equation*}
I_m=\frac{1}{\theta}I_m+\frac{\theta-1}{\theta}I_m<\frac{1}{\theta}\theta
I_{m_1}+\frac{\theta-1}{\theta}\frac{\theta}{\theta-1}I_{m-m_1}=I_{m_1}+I_{m-m_1}.
\end{equation*}
Thus we complete the proof.
\end{proof}

Now we are ready to give the proof of Theorem $\ref{theorem1-3}$.

\begin{proof}[Proof of Theorem  $\ref{theorem1-3}$.]
 Let
$\{\varphi_n\}\in {\ell^2}(\mathbb{Z}^{3})$ be a  minimizing
sequence such that 
$$\|\varphi_n\|_{\ell^2(\mathbb{Z}^{3})}^2=m\quad \mathrm{and}\quad 
\mathop{\lim}\limits_{n\rightarrow\infty}E(\varphi_n)=I_m.$$ 
Since the sequence $\{\varphi_n\}$ is bounded in ${H^1}(\mathbb{Z}^{3})$,
it follows that ${\varphi_n} \rightharpoonup \varphi$ in
${H^1}(\mathbb{Z}^{3})$ and ${\varphi_n} \to \varphi$ pointwise on
$\mathbb{Z}^{3}$. Since the $\ell^2$ norm is weakly lower
semi-continuous, it leads to
$$0\leq\|\varphi\|_{\ell^2(\mathbb{Z}^{3})}^2\le m.$$ 
Next, we divide the proof into the following three steps 
\medskip

\noindent {\bf Step 1}: The sequence $\{\varphi_n\}$ is non-vanishing. By
\eqref{H-2} and the inequality \eqref{p-q}, then
\begin{equation*}
E(\varphi_n)\geq \sum_{y\in\mathbb{Z}^{3}}|\nabla\varphi_n(y)|^2-
\sum_{y\in\mathbb{Z}^{3}} \varphi^{\frac{8}{3}}_n(y) \geq
Cm-m^\frac{4}{3},
\end{equation*}
where $C$ is some positive constant. Since $E'(\varphi_n)\varphi_n=0$, it follows that
\begin{equation*}
\sum_{y\in\mathbb{Z}^{3}}|\nabla\varphi_n(y)|^2+
\frac{10}{3}\sum_{y\in\mathbb{Z}^{3}}\varphi^{\frac{10}{3}}_n(y)
-\frac{8}{3}\sum_{y\in\mathbb{Z}^{3}}\varphi^{\frac{8}{3}}_n(y) +
4D(\varphi_n^2,\varphi_n^2)=0.
\end{equation*}
Thus, by the interpolation inequality, we get
\begin{equation*}
\begin{split}
Cm-m^\frac{4}{3}&\le\lim_{n\to\infty}E(\varphi_n)=\lim_{n\to\infty}\left(-\frac{7}{3}\sum_{y\in\mathbb{Z}^{3}}\varphi^{\frac{10}{3}}_n(y)
+\frac{5}{3}\sum_{y\in\mathbb{Z}^{3}}\varphi^{\frac{8}{3}}_n(y)
-3D(\varphi_n^2,\varphi_n^2)\right)\\
&\le \lim_{n\to\infty}\frac{5}{3}\sum_{y\in\mathbb{Z}^{3}}\varphi^{\frac{8}{3}}_n(y)\le\lim_{n\to\infty}\frac{5}{3}\|\varphi_n\|_{\ell^2(\mathbb{Z}^{3})}^2\|\varphi_n\|_{\ell^\infty(\mathbb{Z}^{3})}^{\frac{2}{3}}=\lim_{n\to\infty}\frac{5}{3}m\|\varphi_n\|_{\ell^\infty(\mathbb{Z}^{3})}^{\frac{2}{3}},
\end{split}
\end{equation*}
which implies
$$\lim\limits_{n\to\infty}\|\varphi_n\|_{\ell^\infty(\mathbb{Z}^{3})}>\frac{1}{2}C>0$$
provided $m$ is small enough. Since
$\|\varphi_n\|_{\ell^2(\mathbb{Z}^{3})}^2=m$, it follows that the
maximum of $|\varphi_n|$ is attainable in a bounded subset of
$\mathbb{Z}^{3}$. We set
$\widetilde{\varphi}_n(y) := \varphi_n(y-{y_n})$, where $y_n$ is chosen such that
$|\varphi_n(y_n)|=\max\limits_{y}|\varphi_n(y)|$. Taking a
subsequence if necessary, we have
$$\widetilde{\varphi}_n \to \widetilde{\varphi} \,\,\,\,\,\, \text{pointwise}\,\,\,\text{in} \,\,\, \mathbb{Z}^{3}$$
and
$$|\widetilde{\varphi}(0)|=\lim_{n\to\infty}\|\widetilde{\varphi}_n\|_{\ell^\infty(\mathbb{Z}^{3})}>0.$$
By translation invariance, we can choose  a minimizing sequence such that the limit $\varphi\ne0$.
\medskip

\noindent {\bf Step 2}: We claim that
$\|\varphi\|_{\ell^2(\mathbb{Z}^{3})}^2=m$. Using a contradiction
argument, suppose that
\begin{equation*}
0<\|\varphi\|_{\ell^2(\mathbb{Z}^{3})}^2=m_1<m.
\end{equation*}
By the  Lemma \ref{lemmaB-L}, we get that
\begin{equation*}
\|\varphi_n-\varphi\|_{\ell^2(\mathbb{Z}^{3})}^2=m-m_1.
\end{equation*}
From Lemmas \ref{lemmaB-L}-\ref{lemma-nonlocal} and
\cite[Corollary 11]{Hua-Li-2021} that
\begin{equation*}
I_m=\lim_{n\rightarrow\infty}E(\varphi_n)=\lim_{n\rightarrow\infty}E(\varphi_n-\varphi)+E(\varphi)\ge
I_{m_1}+I_{m-m_1}.
\end{equation*}
This is in contradiction to the inequality \eqref{q-1}.
\medskip

\noindent {\bf Step 3}: $\varphi$ is the minimizer of the functional $E$, i.e.
$ E(\varphi)=I_m$. According to the infimum  of the $I_m$ and
$\|\varphi\|_{\ell^2(\mathbb{Z}^{3})}^2=m$, then  $ E(\varphi)\ge
I_m$. By the inequality \eqref{p-q}, we get
\begin{equation*}
\|\varphi_n-\varphi\|_{\ell^{\frac{8}{3}}(\mathbb{Z}^{3})}\le\|\varphi_n-\varphi\|_{\ell^2(\mathbb{Z}^{3})}.
\end{equation*}
Then it follows that $\mathop {\lim
}\limits_{n\to\infty}\|\varphi_n\|_{\ell^{\frac{8}{3}}(\mathbb{Z}^{3})}
=\|\varphi\|_{\ell^{\frac{8}{3}}(\mathbb{Z}^{3})}$. Using
the Fatou's Lemma we have
\begin{equation*}
\begin{split}
I_m=\lim_{n\rightarrow\infty}E(\varphi_n)
&=\lim_{n\rightarrow\infty}\left(\sum_{y\in\mathbb{Z}^{3}}\left(|\nabla\varphi_n(y)|^2+
F(\varphi_n^2(y))\right) +
D(\varphi_n^2,\varphi_n^2)\right)\\
&\ge \sum_{y\in\mathbb{Z}^{3}}\left(|\nabla\varphi(y)|^2+
F(\varphi^2(y))\right) + D(\varphi^2,\varphi^2)=E(\varphi).
\end{split}
\end{equation*}
This implies that $ E(\varphi)=I_m$. Therefore, $\varphi$ is a
minimizer of the functional $E$ provided the mass $m$ is small enough. Hence, we finish the proof.
\end{proof}

\section{The nonexistence of a minimizer for \eqref{a-1}}

In this section, we will establish the proof for Theorem \ref{theorem1-2}. The key idea is to demonstrate that the minimizer $\varphi$ of \eqref{a-1} must have a uniform bound in the $l^2(\mathbb{Z}^3)$ norm, i.e., $m:=\sum\limits_{y\in\mathbb{Z}^3}\varphi^2<C$, where $C$ is a finite positive constant. Consequently, if $m>C$, the functional \eqref{a-1} cannot have any minimizer. To begin with, we will examine this fact by analyzing the following elementary property of $\varphi$.

\begin{lemma}\label{lemma4.1}
For every $x$ in $\mathbb{Z}^{3}$, then we have
$\varphi(x)\in[0,(\frac{4}{5})^\frac{3}{2}]$ and
$-\varphi^2(x)\lesssim F({\varphi}^2) \leq0$.
\end{lemma}

\begin{proof}
As $E(|\varphi|)\leq E(\varphi)$, we can focus our analysis on the case where $\varphi\geq0$. Consider the function $F(s)=s^{\frac{5}{3}}-s^{\frac{4}{3}}$, we have $\min\limits_{s\geq0}F(s)=F\left((\frac{4}{5})^3\right)$. We define $\varphi_1=\min\left\{\varphi,(\frac{4}{5})^\frac{3}{2}\right\}$ and $m_1:=\sum\limits_{y\in\mathbb{Z}^{3}}\varphi_1^2(y)$, where it is evident that $m_1\leq m$. Next, we will proceed with our argument in the following three steps:
\medskip

\noindent {\bf Step 1}. We claim that
\begin{equation}
\label{d-1}
E(\varphi)\leq E(\varphi_1)+I_{m-m_1} .
\end{equation}
Indeed, for any $\varepsilon>0$, choosing $\psi$ with
$\sum\limits_{y\in\mathbb{Z}^{3}} \psi^2(y) = m-m_1$ and $I_{m-m_1}\leq
E(\psi) \leq I_{m-m_1}+\varepsilon$. 
We then define
\begin{equation*}
\psi_L(x):=C_L[\varphi_1(x)+\psi(x+ Le_1)],
\end{equation*}
where $e_1:= (1,0,0) \in \mathbb{Z}^{3}$ and $C_L$ is a normalization constant such that 
$$\sum\limits_{y\in\mathbb{Z}^{3}} \psi_L^2(y)=\sum\limits_{y\in\mathbb{Z}^{3}} \varphi^2(y)= m.$$ Moreover, it is not difficult to check that
\begin{equation*}
\begin{split}
\sum\limits_{y\in\mathbb{Z}^{3}}\psi_L^2(y) &
=C_L^2\left(\sum\limits_{y\in\mathbb{Z}^{3}}\varphi_1^2(y)+\sum\limits_{y\in\mathbb{Z}^{3}}\psi^2(y)+2\sum\limits_{y\in\mathbb{Z}^{3}}\varphi_1(y)\psi\left(y+ Le_1\right)\right)\\
&=C_L^2\left(m+2\sum_{y\in\mathbb{Z}^{3}}\varphi_1(y)\psi\left(y+
Le_1\right)\right).
\end{split}
\end{equation*}
On the other hand, it is worth noting that $C_L \to 1$ as $L$ approaches infinity. Consequently
\begin{equation}
\label{l-d-1}
\limsup_{L\to\infty} E(\psi_L)= E(\varphi_1)+E(\psi)\leq
E(\varphi_1)+I_{m-m_1}+\varepsilon.
\end{equation}
Using the fact that $\varphi$ is the minimizer of \eqref{a-1}, we get that $E(\varphi)\leq E(\psi_L)$ for any $L$. Sending $\varepsilon\to0$ in \eqref{l-d-1}, we arrive at \eqref{d-1}.
\medskip

\noindent {\bf Step 2}. In this step we shall show that 
\begin{equation}
\label{4-1}
E(\varphi) \leq E(\varphi_1).
\end{equation}
In order to achieve this inequality, it suffices to verify that
\begin{equation}
	\label{4-2}
I_{m-m_1} = \inf_{\sum\limits_{y\in\mathbb{Z}^{3}}\psi^2(y)=
m-m_1} E(\psi) \leq 0.
\end{equation}
Now we define the function $\psi_n$ in the following way:
\begin{equation}\label{d-2}
\psi_n(x):=
\begin{cases}
(n-\ell+1)d, &\ell\leq n,\\
0, & \ell>n,
\end{cases}
\end{equation}
where
\begin{equation*}
\ell=d(x,0)\quad\text{and}\quad
d^2=\frac{m-m_1}{\sum\limits_{\ell=1}^{n}(n-\ell+1)^2(4\ell^2+2)+(n+1)^2}.
\end{equation*}
To obtain a more precise comprehension of the meaning of $\psi_n$, we provide a schematic diagram in Figure 3.
\begin{figure}[htbp]
  \centering
  \includegraphics[width=0.28\textwidth]{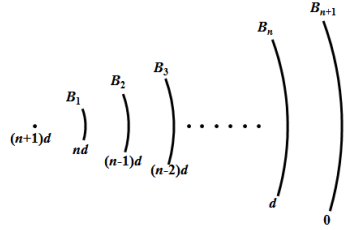}
  \caption{The sketch of $\psi_n$}\label{fig:3}
\end{figure}

\noindent  With the help of Lemma \ref{lemma2-2}, we can easily compute
that
\begin{equation*}
\begin{aligned}
\sum\limits_{y\in\mathbb{Z}^{3}} \psi_n^2(y)
=~&(n+1)^2d^2+\sum_{\ell=1}^{n}\sum_{y\in\partial B_\ell(0)\cap\mathbb{Z}^{3}}\psi_n^2(y)\\
=~&\left[(n+1)^2+\sum_{\ell=1}^{n}(n-\ell+1)^2(4\ell^2+2)\right]d^2=m-m_1.
\end{aligned}
\end{equation*}
Regarding the gradient term, using Lemma \ref{lemma2-2} we have
\begin{equation}
\label{l-d-2}
\begin{aligned}
\sum_{y\in\mathbb{Z}^{3}}|\nabla\psi_n|^2&=\frac12\sum_{y\in\mathbb{Z}^{3}}\sum_{z\sim y,\atop z\in\mathbb{Z}^3}(\psi_n(x)-\psi_n(y))^2\leq Cd^2\sum_{\ell=1}^{n+1}(4\ell^2+2)\\
&=\frac{C(m-m_1)\sum\limits_{\ell=1}^{n+1}(4\ell^2+2)}{\sum\limits_{\ell=1}^{n}(n-\ell+1)^2(4\ell^2+2)+(n+1)^2}.
\end{aligned}
\end{equation}
After direct computation, we get that
\begin{equation*}
a_n:=\sum_{\ell=1}^{n+1}(4\ell^2+2)=\frac23(2n^3+9n^2+16n+9),
\end{equation*}
and
\begin{equation*}
b_n:=\sum_{\ell=1}^{n}(n-\ell+1)^2(4\ell^2+2)+(n+1)^2=
\frac{1}{15}(2n^5+10n^4+30n^3+50n^2+43n+15).
\end{equation*}
Therefore, we have $\lim\limits_{n\to+\infty}\frac{a_n}{b_n}=0$
and it implies that
\begin{equation}
\label{d-3}
\lim_{n\to+\infty}\sum_{y\in\mathbb{Z}^{3}}|\nabla\psi_n(y)|^2 =0.
\end{equation}
While for the nonlinear term, after a straightforward calculation we obtain that
\begin{align*}
\sum_{y\in\mathbb{Z}^{3}}\psi_n^\frac{10}{3}(y)
=&\sum_{\ell=1}^{n}(n-\ell+1)^\frac{10}{3}d^\frac{10}{3}(4\ell^2+2)+(n+1)^\frac{10}{3}d^\frac{10}{3}\\
=&~d^\frac{10}{3}\sum_{\ell=1}^{n}\left((n-\ell+1)^2(4\ell^2+2)^\frac{3}{5}\right)^\frac{5}{3}
+(n+1)^\frac{10}{3}d^\frac{10}{3}\\
\leq&\left(d^2\sum_{\ell=1}^{n}(n-\ell+1)^2(4\ell^2+2)^\frac{3}{5}\right)^\frac{5}{3}
+\left((n+1)^2d^2\right)^\frac{5}{3}\\
=&\left(\frac{(m-m_1)\sum\limits_{\ell=1}^{n}(n-\ell+1)^2(4\ell^2+2)^\frac{3}{5}}{\sum\limits_{\ell=1}^{n}(n-\ell+1)^2(4\ell^2+2)+(n+1)^2}\right)^\frac{3}{5}+\left(\frac{(m-m_1)(n+1)^2}{\sum\limits_{\ell=1}^{n}(n-\ell+1)^2(4\ell^2+2)+(n+1)^2}\right)^\frac{5}{3}.
\end{align*}
We set
\begin{equation*}
c_n:=\sum_{\ell=1}^{n}(n-\ell+1)^2(4\ell^2+2)^\frac{3}{5}\quad\text{and}\quad d_n=(n+1)^2.
\end{equation*}
We notice that
\begin{equation*}
c_n\leq (4n^2+2)^{\frac35}\sum_{\ell=1}^n\ell^2\leq \frac13(n+1)^3(4n^2+2)^{\frac35}.
\end{equation*}
Then one can easily see that
\begin{equation}
\label{l-d-3}
\lim_{n\to \infty} \frac{c_n}{b_n}=0\quad\mathrm{and}\quad \lim_{n\to \infty} \frac{d_n}{b_n}=0.
\end{equation}
This implies that 
$$\sum_{y\in\mathbb{Z}^{3}}\psi_n^\frac{10}{3}(y)\to
0\quad\mathrm{as}\quad n\to \infty.$$ 
Similarly, we can prove
$$\sum_{y\in\mathbb{Z}^{3}}\psi_n^\frac{8}{3}(y)\to
0\quad\mathrm{as}\quad n\to \infty.$$ 
Therefore, 
\begin{equation}\label{d-4}
\lim_{n\to \infty}\sum_{y\in\mathbb{Z}^{3}}F(\psi_n^2(y))=0.
\end{equation}
Finally, we study the nonlocal term. From the
Hardy-Littlewood Soblve inequality \eqref{b-1} we derive that
\begin{equation*}
\sum_{y\in\mathbb{Z}^{3}}\sum_{x\neq y,\atop x\in\mathbb{Z}^3}
\frac{\psi_n^2(y)\psi_n^2(x)}{|y-x|} \leq C
\|\psi_n\|_{\ell^\frac{12}{5}(\mathbb{Z}^{3})}^4.
\end{equation*}
By a direct computation, one has
\begin{equation*}
\begin{split}
\sum_{y\in\mathbb{Z}^{3}}\psi_n^\frac{12}{5}(y)
=&\left(\sum_{\ell=1}^{n}(n-\ell+1)^\frac{12}{5}(4\ell^2+2)+(n+1)^\frac{12}{5}\right)d^\frac{12}{5}\\
\leq&\left(\sum_{\ell=1}^{n}(n-\ell+1)^2(4\ell^2+2)^\frac{5}{6}d^2\right)^\frac{6}{5}
+\left((n+1)^2d^2\right)^\frac{6}{5}\\
=&\left(\frac{(m-m_1)\sum\limits_{\ell=1}^{n}(n-\ell+1)^2(4\ell^2+2)^\frac{5}{6}}{\sum\limits_{\ell=1}^{n}(n-\ell+1)^2(4\ell^2+2)+(n+1)^2}\right)^\frac{6}{5}+\left(\frac{(m-m_1)(n+1)^2}{\sum\limits_{\ell=1}^{n}(n-\ell+1)^2(4\ell^2+2)+(n+1)^2}\right)^\frac{6}{5}.
\end{split}
\end{equation*}
We define
\begin{equation*}
e_n:=\sum_{\ell=1}^{n}(n-\ell+1)^2(4\ell^2+2)^\frac{5}{6}.
\end{equation*}
Obviously, we have
\begin{equation*}
e_n\leq (4n^2+2)^\frac{5}{6}\sum_{\ell=1}^{n}\ell^2\leq \frac13(n+1)^3(4n^2+2)^\frac{5}{6},
\end{equation*}
then
\begin{equation*}
\lim_{n\to \infty} \frac{e_n}{b_n}=0.
\end{equation*}
Together with \eqref{l-d-3} we have
\begin{equation}\label{d-5}
\lim_{n\to \infty}\sum_{y\in\mathbb{Z}^{3}}\sum_{x\neq y,\atop x\in\mathbb{Z}^3}
\frac{\psi_n^2(y)\psi_n^2(x)}{|x-y|}=0.
\end{equation}
From \eqref{d-3}-\eqref{d-5}, one deduces that $E(\psi_n)\to 0$ as
$n\to\infty$, which together with \eqref{d-1} implies that
\begin{equation}\label{d-6}
E(\varphi)\leq E(\varphi_1).
\end{equation}

\noindent {\bf Step 3}. We claim that
$$\varphi(x)\leq\left(\frac{4}{5}\right)^\frac{3}{2},\quad\forall x\in\mathbb{Z}^{3}.$$ 
We prove the above claim by contradiction. Suppose that
there exists a point $x_0 \in \mathbb{Z}^{3} $ such that
$\varphi(x_0)>(\frac{4}{5})^\frac{3}{2}$. Then, one has
\begin{align*}
\sum_{y\in\mathbb{Z}^{3}} F(\varphi_1^2(y))
&=\sum_{\left\{y\in\mathbb{Z}^{3}:\varphi(y)>(4/5)^\frac{3}{2}\right\}}F(\varphi_1^2(y))
+\sum_{\left\{y\in \mathbb{Z}^{3}:\varphi(y)\leq(4/5)^\frac{3}{2}\right\}}F(\varphi_1^2(y))\\
&=\sum_{\left\{y\in\mathbb{Z}^{3}:\varphi(y)>(4/5)^\frac{3}{2}\right\}}F((4/5)^3)
+\sum_{\left\{y\in \mathbb{Z}^{3}:\varphi(y)\leq(4/5)^\frac{3}{2}\right\}}F(\varphi^2(y))\\
&<\sum_{\left\{y\in\mathbb{Z}^{3}:\varphi(y)>(4/5)^\frac{3}{2}\right\}}F(\varphi^2(y))
+\sum_{\left\{y\in \mathbb{Z}^{3}:\varphi(y)\leq(4/5)^\frac{3}{2}\right\}}F(\varphi^2(y))\\&=\sum_{y\in\mathbb{Z}^{3}} F(\varphi^2(y)).
\end{align*}
Given that $|\nabla\varphi_1|\leq |\nabla\varphi|$ and
$D(\varphi_1^2,\varphi_1^2)\leq D(\varphi^2,\varphi^2)$, we can
conclude that $E(\varphi)>E(\varphi_1)$. However, this contradicts
the inequality \eqref{d-6}. Thus, it proves
$0\leq\varphi\leq(\frac{4}{5})^\frac{3}{2}$, and the remaining part
of the lemma can be easily verified.
\end{proof}

The following lemma presents a crucial estimate of the minimizer
$\varphi$.

\begin{lemma}\label{lemma4.3}
Given any $x\in\mathbb{Z}^{3}$ and two radii $r,R>0$ such that
$R>r+1$, the following holds:
\begin{equation*}
\sum_{y\in (B_{r+1}(x)\cup  B_{r}(x))\cap\mathbb{Z}^3}\varphi^2(y)  \geq
\frac{1}{3(R+r)}\sum_{y\in B_{r}(x)\cap\mathbb{Z}^3}\varphi^2(y)  \sum_{y\in(B_R(x)\setminus B_{r+1}(x))\cap\mathbb{Z}^3}\varphi^2(y).
\end{equation*}
\end{lemma}

\begin{proof}
Without loss of generality, we can assume that $x=0$ due to the translational invariance. Let us define the cutoff functions $\chi_1,\chi_2:\mathbb{Z}^3\to[0,1]$ as
\begin{equation*}
\begin{aligned}
\chi_1(x)=\begin{cases}
1,~&d(x,0)\leq r,\\
0,~&d(x,0)\geq r+1,
\end{cases}
\quad 
\chi_2(x)=\begin{cases}
0,~&d(x,0)\leq r,\\
1,~&d(x,0)\geq r+1.
\end{cases}
\end{aligned}
\end{equation*}	
Let $\varphi_1=\chi_1\varphi$ and $\varphi_2=\chi_2\varphi.$ Similar to the proof of \eqref{d-1}, we can obtain the following equality:
\begin{equation*}
E(\varphi)\leq E(\varphi_1)+E(\varphi_2).
\end{equation*}
This implies	
\begin{equation}\label{d-7}
	\begin{split}
		&\sum_{y\in\mathbb{Z}^{3}}
		\left(|\nabla\varphi_1(y)|^2+|\nabla\varphi_2(y)|^2-|\nabla\varphi(y)|^2\right)
		+\sum_{y\in\mathbb{Z}^{3}}\left(F({\varphi^2_1(y)})+F({\varphi^2_2(y)})-F({\varphi^2(y)})\right)\\
		&\geq
		D(\varphi^2,\varphi^2)-D(\varphi_1^2,\varphi_1^2)-D(\varphi_2^2,\varphi_2^2).
	\end{split}
\end{equation}
We shall analyze the equation \eqref{d-7} by studying the gradient term first. If $d(x,0)<r$ or $d(x,0)>r+1$. Then
\begin{equation*} 
|\nabla\varphi_1|^2+|\nabla\varphi_2|^2-|\nabla\varphi|^2=0.
\end{equation*}
If $d(x,0)=r$,  
\begin{equation*}
	\begin{split}
		|\nabla\varphi_1|^2+|\nabla\varphi_2|^2-|\nabla\varphi|^2=\sum_{y\sim x,\atop d(y,0)=r+1}\frac{1}{2}\left(\varphi^2(x)+\varphi^2(y)-(\varphi(y)-\varphi(x))^2\right)=\sum_{y\sim x,\atop d(y,0)=r+1}\varphi(x)\varphi(y).
	\end{split}
\end{equation*}
If $d(x,0)=r+1$, 
\begin{equation*}
\begin{split}
|\nabla\varphi_1|^2+|\nabla\varphi_2|^2-|\nabla\varphi|^2
=\sum_{y\sim x,\atop d(y,0)=r}\frac{1}{2}\left(\varphi^2(y)+\varphi^2(x)-(\varphi(y)-\varphi(x))^2\right)
=\sum_{y\sim x,\atop d(y,0)=r}\varphi(x)\varphi(y).
\end{split}
\end{equation*}
Therefore, from  the above three equations we get that
\begin{equation}\label{d-12}
\begin{split}
\sum_{y\in\mathbb{Z}^3}\left(|\nabla\varphi_1(y)|^2+|\nabla\varphi_2(y)|^2-|\nabla\varphi(y)|^2\right)
=&\sum_{y\in\partial B_{r}\cap\mathbb{Z}^3}\sum_{x\sim y,\atop x\in{\partial
B}_{r+1}\cap\mathbb{Z}^3\hfill}\varphi(x)\varphi(y)\\
&+\sum_{y\in\partial B_{r+1}\cap\mathbb{Z}^3}\sum_{x\sim y,\atop x\in{\partial B}_{r}\cap\mathbb{Z}^3\hfill}\varphi(x)\varphi(y)\\
\leq &~\sum_{y\in(B_{r+1}\setminus B_{r-1})\cap\mathbb{Z}^3}6\varphi^2(y),
\end{split}
\end{equation}
where the number $6$ appears due to the fact that each vertex point has 6 neighboring points at most. For the second term on the left hand side of \eqref{d-7}, one can easily check that
\begin{equation*}
\begin{split}
F(\varphi_1^2)+F(\varphi_2^2)-F(\varphi^2)=0.
\end{split}
\end{equation*}
As a consequence,
\begin{equation}\label{d-13}
\sum_{y\in\mathbb{Z}^{3}}\left(F(\varphi_1^2(y))+F(\varphi_2^2(y))-F(\varphi^2(y))\right)=0.
\end{equation}
Consider the term on the right hand side of \eqref{d-7}, after a direct computation we get that
\begin{equation}\label{d-14}
\begin{split}
D(\varphi^2,\varphi^2)-D(\varphi_1^2,\varphi_1^2)-D(\varphi_2^2,\varphi_2^2)
=~&2\sum_{x\in\mathbb{Z}^{3}}\sum_{y\neq x,\atop y\in\mathbb{Z}^{3}\hfill}\frac{\varphi_1^2(x)\varphi_2^2(y)}{|x-y|} =2\sum_{x\in\mathbb{Z}^{3}}\sum_{y\neq x,\atop y\in\mathbb{Z}^{3}\hfill}\frac{\varphi_1^2(x)\varphi_2^2(y)}{|x-y|}\\
\geq~&2\sum_{x\in B_{r}\cap\mathbb{Z}^3}\sum_{y\neq x,\atop y\in (B_R\setminus
B_{r+1})\cap\mathbb{Z}^3\hfill}\frac{\varphi_1^2(x)\varphi_2^2(y)}{|x-y|}\\
=~&2\sum_{x\in B_{r}\cap\mathbb{Z}^3}\sum_{y\neq x,\atop y\in (B_R\setminus B_{r+1})\cap\mathbb{Z}^3\hfill}\frac{\varphi^2(x)\varphi^2(y)}{|x-y|}\\
\geq~&\frac{2}{r+R}\sum_{x\in B_{r}\cap\mathbb{Z}^3}\varphi^2(x) \sum_{y\in
(B_R\setminus B_{r+1})\cap\mathbb{Z}^3}\varphi^2(y).
\end{split}
\end{equation}
By \eqref{d-7} and \eqref{d-12}-\eqref{d-14}, we complete
the proof.
\end{proof}

The following lemma provides an upper bound for the mass
$m=\sum\limits_{y\in\mathbb{Z}^{3}}\varphi^2(y)$.

\begin{lemma}\label{lemma4.4}
For any ball $B_{R_0}(x)$ with $R_0\geq4$ and
$\sum\limits_{y\in B_{R_0}(x)}\varphi^2(y)\geq C_0$, there exists a universal constant $C_0$ such that:
\begin{equation*}
m=\sum_{y\in \mathbb{Z}^{3}}\varphi^2(y)\leq
2\sum_{y\in B_{2R_0}(x)\cap\mathbb{Z}^3}\varphi^2(y).
\end{equation*}
\end{lemma}

\begin{proof}
Assuming $x=0$ for simplicity of notation, we consider the balls
$B_r$ that are concentric with $B_{R_0}$. For $r\geq 1$, we define
\begin{equation*}
S(r)=\sum_{y\in B_r(x)\cap\mathbb{Z}^3}\varphi^2(y).
\end{equation*}
From Lemma \ref{lemma4.3}, we can deduce that for all $R>r+1$:
\begin{equation*}
S(r+1)-S(r-1)\geq\frac{S(r)(S(R)-S(r+1))}{3(R+r)}.
\end{equation*}
Therefore, for $R\geq 16$ and $r\in[\frac{R}{4}+1,\frac{R}{2}-1]$ (so
that $r<R-2)$, by monotonicity of $S(r)$, one has
\begin{align*}
S(r+1)-S(r-1)
\geq\frac{S(r)(S(R)-S(r+1))}{6R}\geq\frac{S\left(\frac{R}{4}\right)\left(S(R)-S\left(\frac{R}{2}\right)\right)}{6R}.
\end{align*}
Integrating $r$ from $\frac{R}{4}+1$ to $\frac{R}{2}-1$, we get
\begin{align*}
\sum_{\frac{R}{4}+1}^{\frac{R}{2}-1}S(r+1) 
-\sum_{\frac{R}{4}+1}^{\frac{R}{2}-1}S(r-1)=~&
S\left(\frac{R}{2}\right)+S\left(\frac{R}{2}-1\right)-S\left(\frac{R}{4}+1\right)-S\left(\frac{R}{4}\right)\\
\leq~&2S\left(\frac{R}{2}\right)-2S\left(\frac{R}{4}\right).
\end{align*}
Hence we obtain
\begin{equation*}
S\left(\frac{R}{2}\right)-S\left(\frac{R}{4}\right)\geq\frac{S\left(\frac{R}{4}\right)}{C_1}\left(S(R)-S\left(\frac{R}{2}\right)\right)
\end{equation*}
for some universal constant $C_1$. The inequality holds for any
$R\geq 16$. In particular, for any $k\geq0$, we have
\begin{equation*}
\begin{split}
S\left(2^{k-1}R\right)-S\left(2^{k-2}R\right)
&\geq\frac{S\left(2^{k-2}R\right)}{C_1}\left(S\left(2^kR\right)-S\left(2^{k-1}R\right)\right)\\
&\geq\frac{S\left(\frac{R}{4}\right)}{C_1}\left(S\left(2^kR\right)-S\left(2^{k-1}R\right)\right).\\
\end{split}
\end{equation*}
Let us choose $C_0=2C_1$. Then, by taking $R=4R_0\geq16$, we can
state that for all $k\geq0$:
\begin{equation*}
\begin{split}
S\left(2^kR\right)-S\left(\frac{R}{2}\right)
&=\sum_{i=0}^k\left(S(2^iR)-S(2^{i-1}R)\right) \leq \sum_{i=0}^k\left(\frac{C_1}{S(\frac{R}{4})}\right)^{i+1}
\left(S\left(\frac{R}{2}\right)-S\left(\frac{R}{4}\right)\right)\\
&\leq S\left(\frac{R}{2}\right)-S\left(\frac{R}{4}\right).
\end{split}
\end{equation*}
In the last inequality, we utilized the fact that
$S\left(\frac{R}{4}\right)=S\left(R_0\right)\geq C_0=2C_1$.
Then we get that
\begin{equation*}
2S\left(2R_0\right)=2S\left(\frac{R}{2}\right)\geq
S\left(2^kR\right)=\sum_{y\in B_{2^kR}\cap\mathbb{Z}^3}\varphi^2(y)
\end{equation*}
for any $k$. The lemma is proved by sending $k\to\infty$.
\end{proof}

Now we are ready to give the proof of Theorem  $\ref{theorem1-2}$.

\begin{proof}[Proof of Theorem  $\ref{theorem1-2}$.] Let
\begin{equation}\label{d-15}
R_0:=\inf_R\left\{\,\exists x\in\mathbb{Z}^{3} \,\textrm{such that}
\sum_{y\in B_R(x)\cap\mathbb{Z}^3}\varphi^2(y) \geq C_0\right\},
\end{equation}
where $C_0$ is the constant in Lemma \ref{lemma4.4}. It is clear
that $R_0\in(2,\infty)$ for $m>C_0$. By Lemma \ref{lemma4.4}, we
obtain
\begin{equation*}
m\leq2\sum_{y\in B_{2R_0}(x)\cap\mathbb{Z}^3}\varphi^2(y).
\end{equation*}
For the ball $B_{2R_0}$, we can find $n$ points $x_1,...,x_l \in
B_{2R_0}$ such that
\begin{equation*}
B_{2R_0}(x)\subset\bigcup_{i=1}^{l}B_{\left([R_0]-1\right)}(x_i),
\end{equation*}
where $n$ is independent of $R_0$. It follows that
\begin{equation*}
m\leq2\sum_{y\in B_{2R_0}(x)\cap\mathbb{Z}^3}\varphi^2(y)  =2\sum_{y\in B_{2R_0}(x)\cap\mathbb{Z}^3}\varphi^2(y) 
\leq2\sum_{i=1}^{l}\sum_{y\in B_{\left([R_0]-1\right)}(x_i)\cap\mathbb{Z}^3}\varphi^2(y).
\end{equation*}
One deduces from \eqref{d-15} that $m\leq 2nC_0<\infty$.
\end{proof}

\section{The nonexistence of a minimizer for \eqref{a-2}}

In this section, we will prove Theorem \ref{theorem1-1} by contradiction. Suppose that there exists a minimizer $\Omega$, we first demonstrate that the energy satisfies the estimate: 
\begin{equation} \label{5.1} \inf\limits_{{|\Omega|=V}}E(\Omega)\lesssim V. 
\end{equation} 
On the other hand, we shall establish an estimate for the nonlocal term as follows: 
\begin{equation} \label{5.2} 
	\sum\limits_{x,y \in\Omega \hfill\atop y\ne x\hfill} {\frac{1}{|x-y|}}\gtrsim V\log V. 
\end{equation} 
By combining \eqref{5.1} and \eqref{5.2}, we arrive at a contradiction and it leads to the original conclusion. To this end we define
\begin{equation}\label{c-1}
E(V):=\inf\{E(\Omega)\mid|\Omega|=V \}.
\end{equation}
The following lemma presents an inequality of subadditivity for
$E(V)$.

\begin{lemma}\label{lemma3-1}
Assuming $V_0>1$, $V_1>1$, and $V=V_0+V_1$, it follows that
\begin{equation}\label{c-2}
E(V)\leq E(V_0)+E(V_1).
\end{equation}
\end{lemma}

\begin{proof}
By the definition of infimum, for any given $\epsilon\ge0$, there are
bounded sets ${\Omega}_i, i=0,1$, such that 
$$E({\Omega}_i)\leq E(V_i)+\epsilon.\footnote{In the discrete case, even if the set ${E(\Omega),~|\Omega|=V}$ does not have concentrated values, we can still find sets $\Omega_i$ with aforementioned properties. For instance, if $E(V_i)$ is not a concentrated value, we can choose the minimizer with $\epsilon$ set to be $0$. While if $E(V_i)$ is a concentrated value, we can proceed as in the continuous case.}$$ 
Additionally, there exists a large ball $B_R$ such
that ${\Omega}_i\subset B_R(0)$. Let ${\widetilde{\Omega}}$ be defined
as ${\Omega}_0+(d+{\Omega}_1)$, where $d$ is a shift vector with
$|d|>2R$. It should be noted that when $x\in {\Omega}_0$ and $y \in
d+{\Omega}_1$, we have $|x-y|>d-2R>0$. It is evident that $|\partial
\widetilde{\Omega}|=|\partial {\Omega}_0|+|\partial {\Omega}_1|$ and
$|\widetilde{\Omega}|=| {\Omega}_0|+| {\Omega}_1|$ due to the definition
of $\widetilde{\Omega}$. Thus, based on the definition of $E(V)$, we
have:
\begin{equation*}
\begin{split}
E(V)&\leq E(\widetilde{\Omega})
=|\partial {\Omega}_0|+|\partial {\Omega}_1|+\sum_{x,y \in {\Omega}_0 \hfill\atop ~\ y\ne x\hfill}{\frac{1}{|x-y|}} +\sum_{x,y \in d+{\Omega}_1 \hfill\atop ~\ ~\ y\ne
x\hfill}{\frac{1}{|x-y|}}
+2\sum_{x\in {\Omega}_0,~y \in d+{\Omega}_1 \hfill\atop
~\ \quad  y\ne x\hfill}{\frac{1}{|x-y|}}\\
&\leq E({\Omega}_0)+E({\Omega}_1)+ \frac{2V_0V_1}{|d|-2R}\\
&\leq E(V_0)+E(V_1)+2\epsilon +\frac{2V_0V_1}{|d|-2R}.
\end{split}
\end{equation*}
As $|d|$ tends to infinity and $\epsilon$ tends to zero, the
inequality is established.
\end{proof}

The subsequent lemma establishes an estimate for $E(V)$ and elucidates the properties of the minimizer $\Omega$.

\begin{lemma}\label{lemma3-2}
Let $V\gg1$, then
\begin{enumerate}
\item [$(i)$] We have $E(V)\lesssim V$.
\item [$(ii)$] If $\Omega$ is a minimizer of \eqref{a-2}, then $\Omega$ must be a connected set.
\end{enumerate}
\end{lemma}

\begin{proof}
$(i).$ By Lemma \ref{lemma2-2}, 
$$|\partial\Omega|\sim V^\frac{2}{3}~\mbox{for a ball}~\Omega~\mbox{of volume}~V.$$
Using the fact that $|x-y|\geq1$ we have
$$\sum\limits_{x,y \in {\Omega} \hfill\atop ~ y\ne
	x\hfill}{\frac{1}{|x-y|}}\lesssim V^2.$$
Therefore
$$E(\Omega)\lesssim V^2+V^{\frac23},$$
which implies  
\begin{equation}
\label{5-sub}
E(V)\lesssim 1 \quad \mbox{for}\quad 1\leq |V|<2.
\end{equation}
Through the process of induction and utilizing the subadditivity
inequality \eqref{c-2}, we are able to derive that
\begin{equation*}
E(V)\lesssim N\quad \textrm{for}\quad N\leq V<(N+1).
\end{equation*}

\noindent $(ii).$ When $\Omega$ is not a connected domain, we can divide it
into two separate regions $\Omega_1$ and $\Omega_2$ in such a way
that no point $x\in\Omega_1$ is connected to any point
$y\in\Omega_2$ (i.e., $x\nsim y$). Let us define 
$$\widehat{\Omega}={\Omega}_1+(d+{\Omega}_2),$$ 
where $d$ is a shift vector with
$|d|>1$. It is clear that 
$$|\partial \widehat{\Omega}|=|\partial
{\Omega}_1|+|\partial {\Omega}_2|,\quad
|\widehat{\Omega}|=|
{\Omega}_1|+| {\Omega}_2|=V,$$ 
and
$$\sum\limits_{x,y \in {\widehat\Omega} \hfill\atop ~ y\ne
	x\hfill}\frac{1}{|x-y|}<\sum\limits_{x,y \in {\Omega} \hfill\atop ~ y\ne
	x\hfill}\frac{1}{|x-y|}.$$
Therefore, we have $E(\widehat{\Omega})< E(\Omega)$, which contradicts the definition of $\Omega$.
\end{proof}

The following lemma provides lower bound for the nonlocal expression
$\sum\limits_{x,y \in {\Omega} \hfill\atop ~ y\ne
	x\hfill}\frac{1}{|x-y|}$.

\begin{lemma}
\label{le5.3}
Suppose that $V\gg1$ and $\Omega$ minimizes \eqref{a-2}, then the following
estimate holds for the nonlocal expression:
\begin{equation}\label{c-3}
\sum\limits_{x,y \in {\Omega} \hfill\atop ~ y\ne
	x\hfill}\frac{1}{|x-y|}\gtrsim V\log V.
\end{equation}
\end{lemma}

\begin{proof}
Our first assertion is that
\begin{equation}\label{c-4}
|\Omega\cap B_R(x)|\gtrsim R\,\,\,\,\textrm{for all vertics}\,\,
x\in\Omega \,\, \textrm{and radii}\,
1<R<\frac{1}{2}\textrm{diam}(\Omega).
\end{equation}
Let $x$ be a fixed point in $\Omega$ and $1 < R <
\frac{1}{2}\textrm{diam}(\Omega)$ be a chosen radius. By the
definition of the diameter of $\Omega$, we see that $\Omega$ is not entirely
contained within $B_R(x)$. Using the fact that $\Omega$ is connected, we are able to find
a path $\{x_i\}_{i=1}^{R}\in \Omega\cap B_R(x)$  such that $x=x_1$.
This implies that $ |\Omega\cap B_R(x)|\geq R$ and it proves the
assertion \eqref{c-4}. As depicted in Figure \eqref{fig:4}.
\begin{figure}[htbp]
  \centering
  \includegraphics[width=0.3\textwidth]{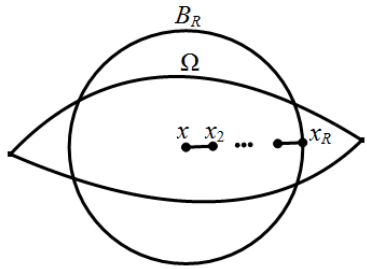}
  \caption{The sketch of $B_R$ and $\Omega$.}\label{fig:4}
\end{figure}

\noindent  Now we claim
\begin{equation*}
\sum\limits_{x,y \in {\Omega} \hfill\atop ~ y\ne
	x\hfill}\frac{1}{|x-y|}\gtrsim V\log\textrm{diam}(\Omega).
\end{equation*}
It is worth noting that, for all
$1<t<\frac{1}{2}\textrm{diam}(\Omega)$,
\begin{equation*}
\begin{split}
A(t):=\big|\{(x,y)\in\Omega\times\Omega ||x-y|<t\}\big|=\sum_{x\in\Omega}\big|\{y\in\Omega ||y-x|<t\}\big|=\sum_{x\in\Omega}\big|\Omega\cap B_{t}(x) \big|\gtrsim t |\Omega|.
\end{split}
\end{equation*}
Hence we get from Lemma
\ref{lemma2-2}~($\textrm{diam}\,\Omega\gtrsim V^\frac{1}{3}$) that
\begin{align*}
\sum\limits_{x,y \in {\Omega} \hfill\atop ~ y\ne
	x\hfill}\frac{1}{|x-y|}
&\geq \sum_{t=1}^{\infty}\frac{1}{t}\big|\{(x,y)\in\Omega\times\Omega | t-1\leq|x-y|<t\}\big|\\
&\geq  \sum_{t=1}^{\infty}\frac{1}{t}(A(t)-A(t-1)) \geq  \sum_{t=1}^{\frac{1}{2}\textrm{diam}(\Omega)}\frac{1}{t(t+1)}A(t)\\
&\geq  \sum_{t=1}^{\frac{1}{2}\textrm{diam}(\Omega)}\frac{1}{2t}|\Omega| \gtrsim |\Omega|\log|\textrm{diam}(\Omega)|\gtrsim V\log V.
\end{align*}
Therefore, we obtain the estimate \eqref{c-3}.
\end{proof}

\begin{proof}[Proof of Theorem \ref{theorem1-1}]
 Theorem \ref{theorem1-1} follows from Lemma \ref{lemma3-2} and Lemma \ref{le5.3}.
\end{proof}

\end{document}